\documentclass[12pt,reqno]{amsart}

\usepackage{amssymb,amsmath,graphicx,amsfonts,euscript}
\usepackage{color}

\newtheorem{thm}{Theorem}[section]

\newtheorem{cor}[thm]{Corollary}
\newtheorem{prop}[thm]{Proposition}
\newtheorem{define}[thm]{Definition}
\newtheorem{rem}[thm]{Remark}

\newtheorem{lemma}[thm]{Lemma}

\newcommand{\R}{\mathbb{R}}
\newcommand{\Z}{\mathbb{Z}}

\numberwithin{equation}{section}
\begin{document}
\title[Global well-posedness to MHD equations]{Global well-posedness to the  3D incompressible MHD equations with a new class of large initial data}

\author[Renhui Wan ]{ Renhui Wan$^{1}$}

\address{$^1$ School of Mathematics Sciences,
Zhejiang University,
Hanzhou 310027, China}

\email{rhwanmath@zju.edu.cn, rhwanmath@163.com}

\vskip .2in
\begin{abstract}
We obtain the global well-posedness to the 3D  incompressible  magnetohydrodynamics (MHD) equations  in
Besov space with negative index of regularity. Particularly, we can get the global solutions for a new class of large initial data. As a byproduct, this  result improves the corresponding  result in  \cite{HHW}. In addition, we also get the global result for this system in  $\mathcal{\chi}^{-1}(\R^3)$ originally developed in \cite{LL}. More precisely, we  only assume that the norm  of  initial data is exactly smaller than the sum of viscosity and diffusivity parameters.

\end{abstract}

\maketitle

\vskip .2in
\section{Introduction}
\label{s1}
We are concerned with the  3D incompressible MHD equations:
\begin{equation} \label{1.1}
\left\{
\begin{aligned}
& \partial_t u + u\cdot\nabla u -\mu_1 \Delta u +\nabla p = B\cdot\nabla B,  \\
& \partial_t B + u\cdot\nabla B-B\cdot\nabla u -\mu_2 \Delta B=0,  \\
& {\rm div} u={\rm div}B=0,\\
& u(0,x) =u_0(x),\ B(0,x)=B_0(x),
\end{aligned}
\right.
\end{equation}
here $(t,x)\in\mathbb{R}^{+}\times\mathbb{R}^3$, $u,p,B$ stand for velocity vector, scalar pressure and magnetic vector, respectively,  $\mu_1$ and $\mu_2$ are  nonnegative viscosity and diffusivity  parameters, respectively.

\vskip .1in
For $\mu_1>0$ and $\mu_2>0$, the local well-posedness and global existence with small data for (\ref{1.1}) were obtained by Duvaut and Lions  \cite{DL} in
$d$  dimensional Sobolev space $H^s(\R^d)$, $s\ge d$. Then Sermange and Termam \cite{ST} studied the regularity of weak solutions $(u,B)\in L^\infty(0,T;H^1(\R^3))$. And some regularity criteria were established in \cite{Wu1,Wu2,Wu3}. For $\mu_1>0$ and $\mu_2=0$ (so-called non-resistive MHD equations), by the new Kato-Ponce commutator estimate,
$$\|\Lambda^s(u\cdot\nabla B)-u\cdot\nabla \Lambda^s B\|_{L^2(\R^d)}\le C\|\nabla u\|_{H^s(\R^d)}\|B\|_{H^s(\R^d)},\ s>\frac{d}{2},\ d=2,3,$$
Fefferman et al. \cite{Fefferman} proved the low regularity  local well-posedness of strong solutions, which was extended to general inhomogeneous Besov space with initial data $(u_0,B_0)\in B_{2,1}^{\frac{d}{2}-1}(\R^d)\times B_{2,1}^\frac{d}{2}(\R^d)$
 in the recent works \cite{Chemin 1} and \cite{Wan}. Furthermore, for the non-resistive version with smooth initial data near some nontrivial steady state, we refer  \cite{PZhnag1,LT,ZZhang,PZhang2} for the related works.

\vskip .1in
Due to a new observation that the velocity field plays a more important role than magnetic field. The new regularity criteria only involving the velocity were proved, see  \cite{CMZ,HX,YZ1,YZ2} and references therein.

\vskip .1in
One can easily get a new formation of (\ref{1.1}) by the following:
$$W^+:=u+B,\ W^-:=u-B,\ \nu_{+}=\frac{\mu_1+\mu_2}{2},\ \nu_{-}=\frac{\mu_1-\mu_2}{2}$$
with initial data $W^{\pm}_0(x):=u_0(x)\pm B_0(x),$ that is,
\begin{equation} \label{1.12}
\left\{
\begin{aligned}
& \partial_t W^+ + W^-\cdot\nabla W^+ -\nu_{+}\Delta W^+ +\nabla p = \nu_{-}\Delta W^-,  \\
& \partial_t W^- + W^+\cdot\nabla W^- -\nu_{+}\Delta W^- +\nabla p = \nu_{-}\Delta W^+,  \\
& {\rm div} W^+={\rm div}W^-=0,\\
& W^+(0,x) =W^+_0(x),\ W^-(0,x)=W^-_0(x).
\end{aligned}
\right.
\end{equation}

Very recently, He et al. \cite{HHW} obtained the global well-posedness for (\ref{1.12}) with initial  data $(u_0,B_0)$ satisfying:\\
($i$) $\nu_-=0$ and
\begin{equation*}
\nu_+^{-3}\|W_0^-\|_{L^3}^3\exp\{C\nu_+^{-3}\|W_0^+\|_{L^3}^3\}<\epsilon_0
\end{equation*}
or
\begin{equation*}
\nu_+^{-3}\|W_0^+\|_{L^3}^3\exp\{C\nu_+^{-3}\|W_0^-\|_{L^3}^3\}<\epsilon_0;
\end{equation*}
($ii$) $\nu_-\neq0$ and
\begin{equation}\label{1.121}
\left(\nu_+^{-2}\|W^-_0\|_{\dot{H}^\frac{1}{2}}^2+\frac{\nu_-^2}{\nu_+^2}(\nu_+^{-2}\|W_0^+\|_{\dot{H}^\frac{1}{2}}^2+\frac{\nu_-^2}{\nu_+^2})\right)
\exp\left\{C\nu_+^{-4}(\|W_0^+\|_{\dot{H}^\frac{1}{2}}^4+\nu_-^4)\right\}<\epsilon_0
\end{equation}
or
\begin{equation*}
\left(\nu_+^{-2}\|W^+_0\|_{\dot{H}^\frac{1}{2}}^2+\frac{\nu_-^2}{\nu_+^2}(\nu_+^{-2}\|W_0^-\|_{\dot{H}^\frac{1}{2}}^2+\frac{\nu_-^2}{\nu_+^2})\right)
\exp\left\{C\nu_+^{-4}(\|W_0^-\|_{\dot{H}^\frac{1}{2}}^4+\nu_-^4)\right\}<\epsilon_0.
\end{equation*}
Here $\epsilon_0$ ia a sufficiently small positive constant.
\vskip .1in

In this paper,  we will prove the global well-posedness of (\ref{1.1}) $(\mu_1>0,\mu_2>0)$ in  generalized space, $\dot{B}_{p,r}^{\frac{3}{p}-1}(\R^3)$,
by make full use of the harmonic analysis tools. The details can be given as follows:

\begin{thm}\label{t1}
Consider (\ref{1.1}) with initial data $(u_0,B_0)\in \dot{B}_{p,r}^{\frac{3}{p}-1}(\R^3),$ $(p,r)\in (1,\infty)\times [1,\infty),$ satisfying ${\rm div}u_0={\rm div}B_0=0$. There exists a constant $C$ and a small constant $\eta>0$ such that if
\begin{equation}\label{1.14}
\left(\|W^-_0\|_{\dot{B}_{p,r}^{\frac{3}{p}-1}}+\frac{\nu_-}{\nu_+}(\|W_0^+\|_{\dot{B}_{p,r}^{\frac{3}{p}-1}}+\nu_-)\right)
\exp\left\{C\nu_+^{-\frac{2}{1-\epsilon}}(\nu_-+\|W_0^+\|_{\dot{B}_{p,r}^{\frac{3}{p}-1}})^\frac{2} {1-\epsilon}          \right\}<\eta \nu_+
\end{equation}
or
\begin{equation}\label{1.15}
\left(\|W^+_0\|_{\dot{B}_{p,r}^{\frac{3}{p}-1}}+\frac{\nu_-}{\nu_+}(\|W_0^-\|_{\dot{B}_{p,r}^{\frac{3}{p}-1}}+\nu_-)\right)
\exp\left\{C\nu_+^{-\frac{2}{1-\epsilon}}(\nu_-+\|W_0^-\|_{\dot{B}_{p,r}^{\frac{3}{p}-1}})^\frac{2}{1-\epsilon}           \right\}<\eta\nu_+,
\end{equation}
where $(\epsilon,r)$ satisfies

\begin{equation} \label{1.16}
\left\{
\begin{aligned}
 0\le \epsilon<1&, \ {\rm if }\  r=1;  \\
 0< \epsilon<1&,  \ {\rm if }\  1<r\le 2;  \\
1-\frac{2}{r}\le \epsilon<1&,\  {\rm if }\  2<r<\infty.
\end{aligned}
\right.
\end{equation}

Then (\ref{1.1}) admits a unique global solution $(u,B)$ satisfying
$$(u,B)\in \tilde{C}([0,\infty); \dot{B}_{p,r}^{\frac{3}{p}-1}(\R^3))\cap \tilde{L}^1([0,\infty); \dot{B}_{p,r}^{\frac{3}{p}+1}(\R^3)).$$
\end{thm}
\vskip .1in
If $\nu_-=0$, i.e., $\mu_1=\mu_2=\nu_+$, we have a corollary immediately.
\begin{cor}\label{c1}
Consider (\ref{1.1}) with initial data $(u_0,B_0)\in \dot{B}_{p,r}^{\frac{3}{p}-1}(\R^3),$ $(p,r)\in (1,\infty)\times [1,\infty),$ satisfying ${\rm div}u_0={\rm div}B_0=0$. There exists a constant $C$ and a small constant $\eta>0$ such that if
\begin{equation}\label{1.141}
\|W^-_0\|_{\dot{B}_{p,r}^{\frac{3}{p}-1}}
\exp\left\{C\nu_+^{-\frac{2}{1-\epsilon}}\|W_0^+\|_{\dot{B}_{p,r}^{\frac{3}{p}-1}}^\frac{2}{1-\epsilon} \right\}<\eta \nu_+
\end{equation}
or
\begin{equation}\label{1.151}
\|W^-_0\|_{\dot{B}_{p,r}^{\frac{3}{p}-1}}
\exp\left\{C\nu_+^{-\frac{2}{1-\epsilon}}\|W_0^+\|_{\dot{B}_{p,r}^{\frac{3}{p}-1}}^\frac{2}{1-\epsilon} \right\}<\eta \nu_+,
\end{equation}
where $(\epsilon,r)$ satisfies   (\ref{1.16}). Then (\ref{1.1}) admits a unique global solution $(u,B)$ satisfying
$$(u,B)\in \tilde{C}([0,\infty); \dot{B}_{p,r}^{\frac{3}{p}-1}(\R^3))\cap \tilde{L}^1([0,\infty); \dot{B}_{p,r}^{\frac{3}{p}+1}(\R^3)).$$
\end{cor}

\begin{rem}\label{r1}
($i$) We will  construct the global  solution with  a new class of large initial data. More precisely,
assume that  $\phi$ satisfies the condition in Proposition \ref{l2}, let
 $$u_0=(\partial_2\phi,-\partial_1\phi,0),\ \ B_0=2\sin^2\frac{x_3}{2\epsilon}(\partial_2\phi,-\partial_1\phi,0),$$
then ${\rm div}u_0={\rm div}B_0=0$ and $\|(u_0,B_0)\|_{\dot{B}_{p,r}^{\frac{3}{p}-1}}\le \mathfrak{M}$ $(p>3)$,  which is independent of $\epsilon$. Moreover,
thanks to  Proposition  \ref{l2},
there exists a positive constant $C_1$ and $C_2$,
$$\|u_0\|_{\dot{B}_{p,r}^{\frac{3}{p}-1}}\ge C_1,\ \ \|B_0\|_{\dot{B}_{p,r}^{\frac{3}{p}-1}}\ge \frac{C_1}{2},$$
$$\|u_0-B_0\|_{\dot{B_{p,r}^{\frac{3}{p}-1}}}\le C_2\epsilon^{1-\frac{3}{p}},$$
which ensures the conditions (\ref{1.14})($\nu_+\gg \nu_-$) and (\ref{1.141}) hold. Additionally,
 the assumption $\nu_+\gg\nu_-$ is reasonable in astrophysical magnetic phenomena, see Remark 2.3 in \cite{HHW}.
Combining with the above explanations, this class of large
data can lead the global well-posedness to (\ref{1.1}). \\
($ii$) One can easily check that condition (\ref{1.14}) is equal to (\ref{1.121}) when $p=r=2$ and choosing $\epsilon=\frac{1}{2}$. By Bernstein's inequality, we have the following embedding relationship:
$$\dot{H}^\frac{1}{2}\hookrightarrow \dot{B}_{p,r}^{\frac{3}{p}-1},\ \ p> 2,  r\ge2.$$
So our result improves the corresponding work under (\ref{1.121}) in \cite{HHW}. By the same way, similar improvements can also  be obtained  under (\ref{1.15}) , (\ref{1.141}) and (\ref{1.151}).
\end{rem}
We shall point out that the above result can not be extended to $p=\infty.$ As a matter of fact, by these works \cite{ill2} and \cite{ill1} concerning
the well-known Navier-Stokes equations, (\ref{1.1}) may ill-posedness in this endpoint Besov space.
\vskip .1in
Next, we  consider the space $\chi^{-1}(\R^3)$, which is smaller than $\dot{B}_{\infty,r}^{-1}$ due to Proposition \ref{l33}.  It was originally  developed in \cite{LL} and applied to get  the global well-posedness for the Navier-Stokes equations under
$$\|u_0\|_{\chi^{-1}}<\mu.$$
For MHD equations (\ref{1.1}), similar result holds under
\begin{equation}\label{1.2}
\|u_0\|_{\chi^{-1}}+\|B_0\|_{\chi^{-1}}<\min\{\mu,\eta\},
\end{equation}
see \cite{KWang} for details.
\vskip .1in
We have some new result  in $\chi^{-1}(\R^3).$

\begin{thm}\label{t2}
Consider (\ref{1.1}) with initial data $(u_0,B_0)\in \chi^{-1}(\R^3)$ satisfying ${\rm div}u_0={\rm div}B_0=0$. There exists a constant $C$ such that
if
\begin{equation}\label{1.21}
\left(\|W_0^-\|_{\chi^{-1}}+\frac{C\nu_-}{\nu_+}(\nu_-+\|W_0^+\|_{\chi^{-1}})\right)\exp\left\{\frac{C}{\nu_+^2}(\nu_-+\|W_0^+\|_{\chi^{-1}})^2\right\}<2\nu_+
\end{equation}
or
\begin{equation}\label{1.22}
\left(\|W_0^+\|_{\chi^{-1}}+\frac{C\nu_-}{\nu_+}(\nu_-+\|W_0^-\|_{\chi^{-1}})\right)\exp\left\{\frac{C}{\nu_+^2}(\nu_-+\|W_0^-\|_{\chi^{-1}})^2\right\} <2\nu_+.
\end{equation}
Then (\ref{1.1}) admits a unique global solution $(u,B)$ satisfying
$$(u,B)\in C([0,\infty); \chi^{-1}(\R^3))\cap L^1([0,\infty);\chi^1(\R^3)).$$
\end{thm}
Similarly, we also have a corollary immediately when $\nu_-=0$.
\begin{cor}\label{c2}
Consider (\ref{1.1}) with initial data $(u_0,B_0)\in \chi^{-1}(\R^3)$ satisfying ${\rm div}u_0={\rm div}B_0=0$. There exists a constant $C$ such that
if
\begin{equation*}
\|W_0^-\|_{\chi^{-1}}\exp\left\{\frac{C}{\nu_+^2}\|W_0^+\|_{\chi^{-1}}^2\right\} <2\nu_+
\end{equation*}
or
\begin{equation*}
\|W_0^+\|_{\chi^{-1}}\exp\left\{\frac{C}{\nu_+^2}\|W_0^-\|_{\chi^{-1}}^2\right\}<2\nu_+.
\end{equation*}
Then (\ref{1.1}) admits a unique global solution $(u,B)$ satisfying
$$(u,B)\in C([0,\infty); \chi^{-1}(\R^3))\cap L^1([0,\infty);\chi^1(\R^3)).$$
\end{cor}
\begin{rem}\label{r2}
The authors in \cite{LL} proved the global well-posedness for Navier-Stokes equations by using
$$\|u\cdot\nabla u\|_{\chi^{-1}}\le \|u\|_{\chi^{-1}}\|u\|_{\chi^1},$$
while we shall use the new estimate below in our  proof, i.e.,
$$\|u\cdot\nabla v\|_{\chi^{-1}}\le \|u\|_{\chi^0}\|v\|_{\chi^0}.$$
\end{rem}
\begin{rem}\label{r3}
Due to the symmetric structure of (\ref{1.12}), we only give the proof of Theorem \ref{t1} and Theorem \ref{t2} under
(\ref{1.14}) and (\ref{1.21}), respectively.
\end{rem}

The present paper is structured as follows:\\
 In section \ref{s2},  we provide some definitions of  spaces, establish several lemmas. The third section proves Theorem \ref{t1}, while the last section  gives the proof of Theorem \ref{t2}.

\vskip .1in
Let us complete this section by describing the notations we shall use in this paper.\\
{\bf Notations}  The uniform constant  $C$ is  different on different lines.  We also use $L^p$,  $\dot{B}_{p,r}^s$ and $\chi^s$  to stand for  $L^p(\mathbb{R}^d)$, $\dot{B}_{p,r}^s(\mathbb{R}^d)$ and $\chi^s(\R^d)$ in somewhere, respectively. We use $A:=B$ to stands for $A$ is defined by $B$,  and ${\bf 1}$ is  the characteristic function.

\vskip .4in
\section{ Preliminaries}
\label{s2}
In this section, we give some necessary definitions, propositions and lemmas.
\vskip .1in
 The Fourier transform is given by
$$\widehat{f}(\xi)=\int_{\R^d}e^{-ix\cdot\xi}f(x)dx.$$

Let $\mathfrak{B}=\{\xi\in\mathbb{R}^d,\ |\xi|\le\frac{4}{3}\}$ and $\mathfrak{C}=\{\xi\in\mathbb{R}^d,\ \frac{3}{4}\le|\xi|\le\frac{8}{3}\}$. Choose two nonnegative smooth radial function $\chi,\ \varphi$ supported, respectively, in $\mathfrak{B}$ and $\mathfrak{C}$ such that
$$\sum_{j\in\mathbb{Z}}\varphi(2^{-j}\xi)=1,\ \ \xi\in\mathbb{R}^d\setminus\{0\}.$$
We denote $\varphi_{j}=\varphi(2^{-j}\xi),$ $h=\mathfrak{F}^{-1}\varphi$ and $\tilde{h}=\mathfrak{F}^{-1}\chi,$ where $\mathfrak{F}^{-1}$ stands for the inverse Fourier transform. Then the dyadic blocks
$\Delta_{j}$ and $S_{j}$ can be defined as follows
$$\Delta_{j}f=\varphi(2^{-j}D)f=2^{jd}\int_{\mathbb{R}^d}h(2^jy)f(x-y)dy,\ \ S_{j}f=\sum_{k\le j-1}\Delta_{k}f$$

Formally, $\Delta_{j}$ is a frequency projection to annulus $\{\xi:\ C_{1}2^j\le|\xi|\le C_{2}2^j\}$, and $S_{j}$ is a frequency projection to the ball $\{\xi:\ |\xi|\le C2^j\}$. One  easily verifies that with our choice of $\varphi$
$$\Delta_{j}\Delta_{k}f=0\ {\rm if} \ |j-k|\ge2\ \ {\rm and}\ \  \Delta_{j}(S_{k-1}f\Delta_{k}f)=0\  {\rm if}\  |j-k|\ge5.$$
Let us recall the definition of the  Besov space.

\begin{define}\label{HB}
 Let $s\in \mathbb{R}$, $(p,q)\in[1,\infty]^2,$ the homogeneous Besov space $\dot{B}_{p,q}^s(\R^d)$ is defined by
$$\dot{B}_{p,q}^{s}(\R^d)=\{f\in \mathfrak{S}'(\R^d);\ \|f\|_{\dot{B}_{p,q}^{s}(\R^d)}<\infty\},$$
where
\begin{equation*}
\|f\|_{\dot{B}_{p,q}^s(\R^d)}=\left\{\begin{aligned}
&\displaystyle (\sum_{j\in \mathbb{Z}}2^{sqj}\|\Delta_{j}f\|_{L^p(\R^d)}^{q})^\frac{1}{q},\ \ \ \ {\rm for} \ \ 1\le q<\infty,\\
&\displaystyle \sup_{j\in\mathbb{Z}}2^{sj}\|\Delta_{j}f\|_{L^p(\R^d)},\ \ \ \ \ \ \ \ {\rm for}\ \ q=\infty,\\
\end{aligned}
\right.
\end{equation*}
and $\mathfrak{S}'(\R^d)$ denotes the dual space of $\mathfrak{S}(\R^d)=\{f\in\mathcal{S}(\mathbb{R}^d);\ \partial^{\alpha}\hat{f}(0)=0;\ \forall\ \alpha\in \ \mathbb{N}^d $\ {\rm multi-index}\} and can be identified by the quotient space of $\mathcal{S'}/\mathcal{P}$ with the polynomials space $\mathcal{P}$.
\end{define}

The norm of the space $\tilde{L}^{r_1}_t(\dot{B}_{p,r}^s)$ and $\tilde{L}^{r_1}_{t,\omega}(\dot{B}_{p,r}^s)$ is defined by
$$\|f\|_{\tilde{L}^{r_1}_t(\dot{B}_{p,r}^s)}:=\|2^{js}\|\Delta_j f\|_{L^{r_1}_tL^p}\|_{l^r(\Z)}$$
and
$$\|f\|_{\tilde{L}^{r_1}_{t,\omega}(\dot{B}_{p,r}^s)}:=\|2^{js}\left(\int_0^t \omega(\tau)^{r_1}\|\Delta_j f(\tau)\|_{L^p}^{r_1} d\tau\right)^\frac{1}{r_1}\|_{l^r(\Z)}.$$
 $f\in \tilde{C}(0,t;\dot{B}_{p,r}^s)$ means $f\in \tilde{L}^{\infty}_t(\dot{B}_{p,r}^s)$ and $\|f(t)\|_{\dot{B}_{p,r}^s}$ is continuous in time.

\vskip .1in
 The following proposition
provide Bernstein type inequalities.
\begin{prop}
Let $1\le p\le q\le \infty$. Then for any $\beta,\gamma\in (\mathbb{N}\cup \{0\})^3$, there exists a constant $C$ independent of $f,j$ such that
\begin{enumerate}
\item[1)] If $f$ satisfies
$$
\mbox{supp}\, \widehat{f} \subset \{\xi\in \mathbb{R}^d: \,\, |\xi|
\le \mathcal{K} 2^j \},
$$
 then
$$
\|\partial^\gamma f\|_{L^q(\mathbb{R}^d)} \le C 2^{j|\gamma|  +
j d(\frac{1}{p}-\frac{1}{q})} \|f\|_{L^p(\mathbb{R}^d)}.
$$
\item[2)] If $f$ satisfies
\begin{equation*}\label{spp}
\mbox{supp}\, \widehat{f} \subset \{\xi\in \mathbb{R}^d: \,\, \mathcal{K}_12^j
\le |\xi| \le \mathcal{K}_2 2^j \}
\end{equation*}
 then
$$
 \|f\|_{L^p(\mathbb{R}^d)} \le C2^{-j|\gamma|}\sup_{|\beta|=|\gamma|} \|\partial^\beta f\|_{L^p(\mathbb{R}^d)}.
$$
\end{enumerate}
\end{prop}
For more details about Besov space  such as  some useful embedding relations, see \cite{BCD,Grafakos,Stein}.

\begin{lemma}\label{l1}\cite{Danchin}
Let $1<p<\infty,$ $supp \widehat{u}\subset C(0,R_{1},R_{2})$ (with $0<R_{1}<R_{2}$). There exists a constant $c$ depending on $\frac{R_{2}}{R_{1}}$ and such that
\begin{equation}\label{Ber}
c\frac{R_{1}^2}{p^2}\int_{\mathbb{R}^3}|u|^p dx\le -\frac{1}{p-1}\int_{\mathbb{R}^3} \Delta u |u|^{p-2}u dx.
\end{equation}
\end{lemma}

\begin{prop}\label{l2}
Let $\phi\in \mathcal{S}(\R^3)$, whose Fourier transform supported in annulus contained in $\R^3\setminus\{0\}$, and $p>3$. If $u_0=(\partial_2 \phi, -\partial_1\phi,0)$ and
$B_0=2\sin^2\frac{x_3}{2\epsilon}(\partial_2 \phi, -\partial_1\phi,0),$ then there exists a constant $C_1,C_2>0$ such that
$$\|u_0\|_{\dot{B}_{p,r}^{\frac{3}{p}-1}}\ge C_1,\ \ \|B_0\|_{\dot{B}_{p,r}^{\frac{3}{p}-1}}\ge \frac{C_1}{2}$$
and
$$\|u_0-B_0\|_{\dot{B}_{p,r}^{\frac{3}{p}-1}}\le C_2\epsilon^{1-\frac{3}{p}},$$
here $\epsilon$ is sufficiently small.
\end{prop}

\begin{proof}
The last estimate can be obtained by following  the proof of Lemma 3.1 in \cite{Chemin 2}. So we suffice to show both
$\|u_0\|_{\dot{B}_{p,r}^{\frac{3}{p}-1}}$ and $\|B_0\|_{\dot{B}_{p,r}^{\frac{3}{p}-1}}$ has positive lower bound. With this $\phi$, there exists a finite
$j_0\in\Z$, such that $\Delta_{j_0}\partial_2\phi\neq0$, which implies
$$\|\Delta_{j_0} \partial_2\phi\|_{L^\infty}\ge \epsilon_0$$
for some positive constant $\epsilon_0$. Thanks to this, by Bernstein's inequality, we have
$$\|u_0\|_{\dot{B}_{p,r}^{\frac{3}{p}-1}}\ge \|u_0\|_{\dot{B}_{\infty,\infty}^{-1}}\ge 2^{-j_0}\|\Delta_{j_0}\partial_2\phi\|_{L^\infty}\ge 2^{-j_0}\epsilon_0,$$
and by triangle inequality
$$\|B_0\|_{\dot{B}_{p,r}^{\frac{3}{p}-1}}\ge \|u_0\|_{\dot{B}_{p,r}^{\frac{3}{p}-1}}-\|u_0-B_0\|_{\dot{B}_{p,r}^{\frac{3}{p}-1}}\ge
2^{-j_0}\epsilon_0-C_2\epsilon^{1-\frac{3}{p}}\ge 2^{-j_0-1}\epsilon_0
$$
due to the sufficient small $\epsilon$. Choosing $C_1=2^{-j_{0}}\epsilon_0$ yields the desired result.
\end{proof}

For some convenience, we provide the following definition of $\chi^{s}(\R^d)$,
$$\|f\|_{\chi^s}:=\int_{\R^d} |\xi|^s |\hat{f}(\xi)| d\xi,$$
and we refer \cite{LL} for some details.

\begin{prop}\label{l33}
Let $f\in \chi^{-1}$, then we have
$$\|f\|_{\dot{B}_{\infty,r}^{-1}}\le \|f\|_{\dot{B}_{\infty,1}^{-1}}\le \|f\|_{\mathbb{{B}}_{1,1}^{-1}}\thickapprox \|f\|_{\chi^{-1}},$$
where
$$\|f\|_{\mathbb{B}_{1,1}^{-1}}:=\sum_{j\in\Z}2^{-j}\|\widehat{\Delta_j f}\|_{L^1}.$$
\end{prop}

\begin{proof}
The first inequality is obvious, while the second inequality can be proved by using
$\|f\|_{L^\infty}\le \|\hat{f}\|_{L^1}.$ Now, we prove  $\|f\|_{\mathbb{{B}}_{1,1}^{-1}}\thickapprox \|f\|_{\chi^{-1}}$.
By the definition of $\Delta_j$,  and using Monotone Convergence Theorem,
\begin{equation*}
\begin{aligned}
\|f\|_{\mathbb{{B}}_{1,1}^{-1}}=&\sum_{j\in \Z}2^{-j}\| \varphi(2^{-j}\xi)\hat{f}(\xi)\|_{L^1}\\
\thickapprox&
\sum_{j\in \Z}\| |\xi|^{-1}\varphi(2^{-j}\xi)\hat{f}(\xi)\|_{L^1}\\
=&\||\xi|^{-1}|\hat{f}(\xi)|\|_{L^1}\\
=&\|f\|_{\chi^{-1}},
\end{aligned}
\end{equation*}
where we have used $\sum_{j\in Z}\varphi(2^{-j}\xi)=1$ and $\varphi\ge 0$.
\end{proof}

\begin{lemma}\label{l3}
($i$) Let $(p,r)\in [1,\infty)\times[1,\infty],$ ${\rm div}u=0$, then
\begin{equation}\label{SKP1}
\|u\cdot\nabla v\|_{\tilde{L}^1_t(\dot{B}_{p,r}^{\frac{3}{p}-1})}\le C\left(\|u\|_{\tilde{L}^\infty_t(\dot{B}_{p,r}^{\frac{3}{p}-1})}\|v\|_{\tilde{L}^1_t(\dot{B}_{p,r}^{\frac{3}{p}+1})}+
\|v\|_{\tilde{L}^\infty_t(\dot{B}_{p,r}^{\frac{3}{p}-1})}\|u\|_{\tilde{L}^1_t(\dot{B}_{p,r}^{\frac{3}{p}+1})}\right);
\end{equation}
($ii$) Let $(p,r)\in [1,\infty)\times[1,\infty],$ ${\rm div}u=0$, then
\begin{equation}\label{SKP2}
\|u\cdot\nabla v\|_{\tilde{L}^1_t(\dot{B}_{p,r}^{\frac{3}{p}-1})}\le
C\|v\|_{\tilde{L}^1_t(\dot{B}_{p,r}^{\frac{3}{p}+1})}^\frac{1+\epsilon}{2}\|v\|_{\tilde{L}^1_{t,f}(\dot{B}_{p,r}^{\frac{3}{p}-1})}^\frac{1-\epsilon}{2},
\end{equation}
where $0<\epsilon<1$ and $f=\|u\|_{\dot{B}_{p,\infty}^{\frac{3}{p}-1}}^\frac{2}{1-\epsilon}.$ In particular, (\ref{SKP2}) also holds when
$(\epsilon,r)=(0,1)$ and $f=\|u\|_{\dot{B}_{p,1}^\frac{3}{p}}^2.$
\end{lemma}
For the proof, we shall use homogeneous Bony's decomposition:
$$uv=T_uv+T_vu+R(u,v),$$
where
$$T_uv=\sum_{j\in \mathbb{Z}}S_{j-1}u\Delta_{j}v,\ \ T_vu=\sum_{j\in \mathbb{Z}}\Delta_{j}uS_{j-1}v,\ \
R(u,v)=\sum_{j\in\mathbb{Z}}\Delta_{j}u\tilde{\Delta}_{j}v,$$
here $\tilde{\Delta}_{j}=\Delta_{j-1}+\Delta_{j}+\Delta_{j+1}.$
\begin{proof}
The estimate of (\ref{SKP1}) can be established by using
$$\|u\cdot\nabla v\|_{\dot{B}_{p,r}^{\frac{3}{p}-1}}\le C\{\|u\|_{\dot{B}_{p,r}^{\frac{3}{p}-1}}\|v\|_{\dot{B}_{p,r}^{\frac{3}{p}+1}}+
\|v\|_{\dot{B}_{p,r}^{\frac{3}{p}-1}}\|u\|_{\dot{B}_{p,r}^{\frac{3}{p}+1}}\},
$$
whose proof is standard. Thus the goal  is the estimate of (\ref{SKP2}). By homogeneous Bony's decomposition,
\begin{equation}\label{2.4}
\begin{aligned}
\|u\cdot\nabla v\|_{\tilde{L}^1_t(\dot{B}_{p,r}^{\frac{3}{p}-1})}\le& \|T_{u_i}\partial_i v\|_{\tilde{L}^1_t(\dot{B}_{p,r}^{\frac{3}{p}-1})}
+\|T_{\partial_i v}u_i\|_{\tilde{L}^1_t(\dot{B}_{p,r}^{\frac{3}{p}-1})}+\|R(u,\nabla v)\|_{\tilde{L}^1_t(\dot{B}_{p,r}^{\frac{3}{p}-1})}\\
:=&I_1+I_2+I_3.
\end{aligned}
\end{equation}
Let $\theta=\frac{1-\epsilon}{2}$, $0<\epsilon<1$. For $I_1$, using H\"{o}lder's inequality and Bernstein's inequality,
\begin{equation*}
\begin{aligned}
I_1\le& \left\|2^{j(\frac{3}{p}-1)}\sum_{|k-j|\le 4}\|\Delta_j(S_{k-1}u\cdot\nabla \Delta_k v)\|_{L^1_tL^p}\right\|_{l^r(\Z)}\\
\le& C\left\|2^{j(\frac{3}{p}-1)}\|S_{j-1}u\cdot\nabla \Delta_j v\|_{L^1_tL^p}\right\|_{l^r(\Z)}\\
\le& C\left\|2^{j(\frac{3}{p}-1)}\int_0^t \|S_{j-1}u\|_{L^\infty}\|\nabla \Delta_j v\|_{L^p}d\tau\right\|_{l^r(\Z)}\\
\le& C\left\|2^{j(\frac{3}{p}+\epsilon)}\int_0^t \|u\|_{\dot{B}_{\infty,\infty}^{-\epsilon}}\| \Delta_j v\|_{L^p}d\tau\right\|_{l^r(\Z)}\\
\le& C\left\|2^{j(\frac{3}{p}+\epsilon)}\int_0^t \|u\|_{\dot{B}_{p,\infty}^{\frac{3}{p}-\epsilon}}\| \Delta_j v\|_{L^p}d\tau\right\|_{l^r(\Z)}\\
\le& C\left\|2^{j(\frac{3}{p}+1)(1-\theta)}\|\Delta_j v\|_{L^1_tL^p}^{1-\theta}
(\int_0^t 2^{j(\frac{3}{p}-1)} \|u\|_{\dot{B}_{p,\infty}^{\frac{3}{p}-\epsilon}}^\frac{1}{\theta}\| \Delta_j v\|_{L^p}d\tau)^\theta\right\|_{l^r(\Z)}\\
\le& C\|v\|_{\tilde{L}^1_t(\dot{B}_{p,r}^{\frac{3}{p}+1})}^{1-\theta}\|v\|_{\tilde{L}^1_{t,f}(\dot{B}_{p,r}^{\frac{3}{p}-1})}^\theta
= C\|v\|_{\tilde{L}^1_t(\dot{B}_{p,r}^{\frac{3}{p}+1})}^\frac{1+\epsilon}{2}\|v\|_{\tilde{L}^1_{t,f}(\dot{B}_{p,r}^{\frac{3}{p}-1})}^\frac{1-\epsilon}{2},
\end{aligned}
\end{equation*}
here $f=\|u\|_{\dot{B}_{p,\infty}^{\frac{3}{p}-1}}^\frac{2}{1-\epsilon}$ and we have used
$$\left\|2^{js}\|S_j u\|_{L^p}\right\|_{l^r(\Z)}\approx \|u\|_{\dot{B}_{p,r}^{s}},\ \forall\ s<0.$$
Similarly, for $I_2$, by H\"{o}lder's inequality and Bernstein's inequality,
\begin{equation*}
\begin{aligned}
I_2\le& \left\|2^{j(\frac{3}{p}-1)}\sum_{|k-j|\le 4}\|\Delta_j(\Delta_{k}u\cdot\nabla S_{k-1} v)\|_{L^1_tL^p}\right\|_{l^r(\Z)}\\
\le& C\left\|2^{j(\frac{3}{p}-1)}\|\Delta_{j}u\cdot\nabla S_{j-1} v\|_{L^1_tL^p}\right\|_{l^r(\Z)}\\
\le& C\left\|2^{j(\frac{3}{p}-1)}\int_0^t \|\Delta_{j}u\|_{L^p}\|\nabla S_{j-1} v\|_{L^\infty} d\tau\right\|_{l^r(\Z)}\\
\le& C\left\|2^{j(\epsilon-1)}\int_0^t \|u\|_{\dot{B}_{p,\infty}^{\frac{3}{p}-\epsilon}}\sum_{j'\le j-2}2^{j'(\frac{3}{p}+1)}\| \Delta_{j'} v\|_{L^p} d\tau\right\|_{l^r(\Z)}\\
\le& C\left\|\sum_{j'\le j-2}2^{(j-j')(\epsilon-1)}\int_0^t \|u\|_{\dot{B}_{p,\infty}^{\frac{3}{p}-\epsilon}}2^{j'(\frac{3}{p}+\epsilon)}\| \Delta_{j'} v\|_{L^p} d\tau\right\|_{l^r(\Z)}\\
\le& C \left\|2^{j(\frac{3}{p}+\epsilon)}\int_0^t \|u\|_{\dot{B}_{p,\infty}^{\frac{3}{p}-\epsilon}}\| \Delta_{j} v\|_{L^p} d\tau\right\|_{l^r(\Z)}\\
\end{aligned}
\end{equation*}
where we have used Young's inequality for series for the last inequality, i.e.,
$$\left\|\sum_{j'\le j-2}2^{(j-j')(\epsilon-1)}c_{j'}\right\|_{l^r(\Z)}\le C\|2^{j(\epsilon-1)}{\bf 1}_{j\ge 2}\|_{l^1(\Z)}\|c_j\|_{l^r(\Z)}\le C\|c_j\|_{l^r(\Z)}.$$
Following the same argument as $I_1$, one gets
$$I_2\le C\|v\|_{\tilde{L}^1_t(\dot{B}_{p,r}^{\frac{3}{p}+1})}^\frac{1+\epsilon}{2}\|v\|_{\tilde{L}^1_{t,f}(\dot{B}_{p,r}^{\frac{3}{p}-1})}^\frac{1-\epsilon}{2}.$$
Finally, we bound $I_3$. By Bernstein's inequality, Young's inequality for series and H\"{o}lder's inequality, we have
\begin{equation*}
\begin{aligned}
I_3\le& \left\|2^{j(\frac{3}{p}-1)}\sum_{k\ge j-3}\|\Delta_j(\Delta_{k}u\cdot\nabla \tilde{\Delta}_k v)\|_{L^1_tL^p}\right\|_{l^r(\Z)}\\
\le& C\left\|2^{j\frac{3}{p}}\sum_{k\ge j-3}\|\Delta_j(\Delta_{k}u\otimes \tilde{\Delta}_k v)\|_{L^1_tL^p}\right\|_{l^r(\Z)}\\
\le& C\left\|\sum_{k\ge j-3}2^{(j-k)\frac{3}{p}}2^{k\frac{3}{p}}\|\Delta_j(\Delta_{k}u\otimes \tilde{\Delta}_k v)\|_{L^1_tL^p}\right\|_{l^r(\Z)}\\
\le& C\left\|2^{k\frac{3}{p}}\int_0^t \|\Delta_k u\|_{L^p}\|\tilde{\Delta}_k v\|_{L^\infty} d\tau \right\|_{l^r(\Z)}\ (p<\infty)\\
\le& C\left\|2^{k(\frac{3}{p}+\epsilon)}\int_0^t \| u\|_{\dot{B}_{p,\infty}^{\frac{3}{p}-\epsilon}}\|\tilde{\Delta}_k v\|_{L^p} d\tau \right\|_{l^r(\Z)},
\end{aligned}
\end{equation*}
and using the same way as the estimate of $I_1$ derives
$$I_3\le C\|v\|_{\tilde{L}^1_t(\dot{B}_{p,r}^{\frac{3}{p}+1})}^\frac{1+\epsilon}{2}\|v\|_{\tilde{L}^1_{t,f}(\dot{B}_{p,r}^{\frac{3}{p}-1})}^\frac{1-\epsilon}{2}.$$
Plugging the above estimates into (\ref{2.4}) leads the desired result (\ref{SKP2}).\\

In addition, if $r=1$, the estimate of $I_1$ can be replaced as follows:
\begin{equation*}
\begin{aligned}
I_1\le& C\left\|2^{j(\frac{3}{p}-1)}\int_0^t \|S_{j-1}u\|_{L^\infty}\|\nabla \Delta_j v\|_{L^p}d\tau\right\|_{l^r(\Z)}\\
\le& C\left\|2^{j\frac{3}{p}}\int_0^t \|u\|_{\dot{B}_{p,1}^\frac{3}{p}}\|\Delta_j v\|_{L^p}d\tau\right\|_{l^r(\Z)}\\
\le& C\left\|2^{j(\frac{3}{p}+1)\frac{1}{2}}\|\Delta_j v\|_{L^1_tL^p}^\frac{1}{2}(2^{j(\frac{3}{p}-1)}\int_0^t \|u\|_{\dot{B}_{p,1}^\frac{3}{p}}^2\|\Delta_j v\|_{L^p}d\tau)^\frac{1}{2}\right\|_{l^r}\\
\le& C\|v\|_{\tilde{L}^1_t(\dot{B}_{p,r}^{\frac{3}{p}+1})}^\frac{1}{2}\|v\|_{\tilde{L}^1_{t,f}(\dot{B}_{p,r}^{\frac{3}{p}-1})}^\frac{1}{2},
\end{aligned}
\end{equation*}
here $f=\|v\|_{\dot{B}_{p,1}^\frac{3}{p}}^2$. At the same time, one can get the new estimates of $I_2$ and $I_3$ with the similar procedure.
Thus we complete the proof of this lemma.
\end{proof}

\vskip .3in
\section{Proof of Theorem \ref{t1}}
\label{s3}

As the Remark \ref{r3}, it suffices to prove the Theorem \ref{t1} under (\ref{1.14}). One can get the local existence and uniqueness for
(\ref{1.1}) by using the standard argument on the Navier-Stokes equations, namely, there exists a $T^\star>0$, such that
$$(u,B)\in \tilde{C}([0,T^\star); \dot{B}_{p,r}^{\frac{3}{p}-1})\cap \tilde{L}^1([0,T^\star); \dot{B}_{p,r}^{\frac{3}{p}+1}).$$
Since the equivalence between (\ref{1.1}) and (\ref{1.12}), we will consider (\ref{1.12}) and suffice to prove $T^\star=\infty$.\\

Now, we begin the proof. Let us consider $0<\epsilon<1$ and $r\le \frac{2}{1-\epsilon}$, containing  all cases in (\ref{1.16})   except $(\epsilon,r)=(0,1).$ Define
\begin{equation}\label{3.1}
\bar{T}:=\sup\left\{t\in (0,T^\star):\ \|W^-\|_{\tilde{L}^\infty_t(\dot{B}_{p,r}^{\frac{3}{p}-1})}+\nu_+\|W^-\|_{\tilde{L}^1_t(\dot{B}_{p,r}^{\frac{3}{p}+1})}\le \epsilon_o \nu_+ \right\},
\end{equation}
where $\epsilon_0$ is small positive constant and will be determined later on.\\

{\bf Step 1. The estimate of $W^+$.}\ Consider the first equation in (\ref{1.12}), using (\ref{Ber}), we get
$$\frac{d}{dt}\|\Delta_j W^+\|_{L^p}+c\nu_+ 2^{2j}\|\Delta_j W^+\|_{L^p}\le C\|\Delta_j (W^-\cdot\nabla W^+)\|_{L^p}+C\nu_- 2^{2j}\|\Delta_j W^-\|_{L^p},$$
which yields by a standard procedure
\begin{equation*}
\begin{aligned}
\|W^+\|_{\tilde{L}^\infty_t(\dot{B}_{p,r}^{\frac{3}{p}-1})}&+c\nu_+\|W^+\|_{\tilde{L}^1_t(\dot{B}_{p,r}^{\frac{3}{p}+1})}\\
\le& 2\|W^+_0\|_{\dot{B}_{p,r}^{\frac{3}{p}-1}}+C\|W^-\cdot\nabla W^+\|_{\tilde{L}^1_t(\dot{B}_{p,r}^{\frac{3}{p}-1})}
+C\nu_-\|W^-\|_{\tilde{L}^1_t(\dot{B}_{p,r}^{\frac{3}{p}+1})}.
\end{aligned}
\end{equation*}
By (\ref{SKP1}) and (\ref{3.1}), we have for all $t\in (0,\bar{T}],$
\begin{equation*}
\begin{aligned}
\|W^+\|_{\tilde{L}^\infty_t(\dot{B}_{p,r}^{\frac{3}{p}-1})}+&c\nu_+\|W^+\|_{\tilde{L}^1_t(\dot{B}_{p,r}^{\frac{3}{p}+1})}
\le 2\|W^+_0\|_{\dot{B}_{p,r}^{\frac{3}{p}-1}}+C\nu_-\|W^-\|_{\tilde{L}^1_t(\dot{B}_{p,r}^{\frac{3}{p}+1})}\\
&+C(\|W^-\|_{\tilde{L}^\infty_t(\dot{B}_{p,r}^{\frac{3}{p}-1})}
\|W^+\|_{\tilde{L}^1_t(\dot{B}_{p,r}^{\frac{3}{p}+1})}+\|W^+\|_{\tilde{L}^\infty_t(\dot{B}_{p,r}^{\frac{3}{p}-1})}
\|W^-\|_{\tilde{L}^1_t(\dot{B}_{p,r}^{\frac{3}{p}+1})})\\
\le& 2\|W^+_0\|_{\dot{B}_{p,r}^{\frac{3}{p}-1}}+C\epsilon_0 ( \nu_+\|W^+\|_{\tilde{L}^1_t(\dot{B}_{p,r}^{\frac{3}{p}+1})}+
\|W^+\|_{\tilde{L}^\infty_t(\dot{B}_{p,r}^{\frac{3}{p}-1})})+C\epsilon_0\nu_-,\\
\end{aligned}
\end{equation*}
with the selection of $\epsilon_0<\min\{\frac{c}{2C},\frac{1}{2C}\}$ leads
\begin{equation}\label{3.2}
\|W^+\|_{\tilde{L}^\infty_t(\dot{B}_{p,r}^{\frac{3}{p}-1})}+c\nu_+\|W^+\|_{\tilde{L}^1_t(\dot{B}_{p,r}^{\frac{3}{p}+1})}
\le 4\|W^+_0\|_{\dot{B}_{p,r}^{\frac{3}{p}-1}}+2c\nu_-.
\end{equation}

{\bf Step 2. The estimate of $W^-$.}\ Denote
$$f(t):=\|W^+(t)\|_{\dot{B}_{p,\infty}^{\frac{3}{p}-\epsilon}}^\frac{2}{1-\epsilon},\ W_\lambda^\pm:=W^\pm\exp\{-\lambda\int_0^t f(\tau) d\tau\},
\ p_\lambda:=p\exp\{-\lambda\int_0^t f(\tau) d\tau\},$$
where $\lambda$ is large enough constant and will be determined later on. So we can rewrite the second equation in (\ref{1.12}) as
$$\partial_t W^-_\lambda+\lambda f(t)W^-_\lambda+W^+\cdot\nabla W^-_\lambda+\nabla p_\lambda-\nu_+\Delta W^-_\lambda=\nu_-\Delta W^+_\lambda.$$
By a similar procedure, we have
\begin{equation*}
\begin{aligned}
\|\Delta_j W^-_\lambda\|_{L^\infty_t L^p}+&\lambda \int_0^t f(\tau)\|\Delta_j W^-_\lambda\|_{L^p} d\tau+c\nu_+2^{2j}\|\Delta_j W^-_\lambda\|_{L^1_tL^p}\\
&\le \|\Delta_j W_0^-\|_{L^p}+C\|\Delta_j(W^+\cdot\nabla W^-_\lambda)\|_{L^1_tL^p}+C\nu_-2^{2j}\|\Delta_j W^+_\lambda\|_{L^1_tL^p}.
\end{aligned}
\end{equation*}
Then we obtain
\begin{equation*}
\begin{aligned}
\|W^-_\lambda\|_{\tilde{L}^\infty_t(\dot{B}_{p,r}^{\frac{3}{p}-1})}+&c\nu_+\|W^-_\lambda\|_{\tilde{L}^1_t(\dot{B}_{p,r}^{\frac{3}{p}+1})}
+\lambda \|W^-_\lambda\|_{\tilde{L}^1_{t,f}(\dot{B}_{p,r}^{\frac{3}{p}-1})}\\
&\le \|W_0^-\|_{\dot{B}_{p,r}^{\frac{3}{p}-1}}+C\nu_-\|W^+_\lambda\|_{\tilde{L}^1_t(\dot{B}_{p,r}^{\frac{3}{p}+1})}
+C\|W^+\cdot\nabla W^-_\lambda\|_{\tilde{L}^1_t(\dot{B}_{p,r}^{\frac{3}{p}-1})}.
\end{aligned}
\end{equation*}
Thanks to (\ref{SKP2}), and by Young's inequality, we obtain
\begin{equation}\label{3.3}
\begin{aligned}
&\ \ \ \|W^-_\lambda\|_{\tilde{L}^\infty_t(\dot{B}_{p,r}^{\frac{3}{p}-1})}+c\nu_+\|W^-_\lambda\|_{\tilde{L}^1_t(\dot{B}_{p,r}^{\frac{3}{p}+1})}
+\lambda \|W^-_\lambda\|_{\tilde{L}^1_{t,f}(\dot{B}_{p,r}^{\frac{3}{p}-1})}\\
\le& \|W_0^-\|_{\dot{B}_{p,r}^{\frac{3}{p}-1}}+C\nu_-\|W^+_\lambda\|_{\tilde{L}^1_t(\dot{B}_{p,r}^{\frac{3}{p}+1})}
+C\|W^-_\lambda\|_{\tilde{L}^1_t(\dot{B}_{p,r}^{\frac{3}{p}+1})}^\frac{1+\epsilon}{2}
\|W^-_\lambda\|_{\tilde{L}^1_{t,f}(\dot{B}_{p,r}^{\frac{3}{p}-1})}^\frac{1-\epsilon}{2}\\
\le& \|W_0^-\|_{\dot{B}_{p,r}^{\frac{3}{p}-1}}+C\nu_-\|W^+_\lambda\|_{\tilde{L}^1_t(\dot{B}_{p,r}^{\frac{3}{p}+1})}
+\frac{c\nu_+}{8}\|W^-_\lambda\|_{\tilde{L}^1_t(\dot{B}_{p,r}^{\frac{3}{p}+1})}\\
&+C\nu_+^{-\frac{1+\epsilon}{1-\epsilon}}\|W^-_\lambda\|_{\tilde{L}^1_{t,f}(\dot{B}_{p,r}^{\frac{3}{p}-1})}.
\end{aligned}
\end{equation}
Choosing $\lambda>2C\nu_+^{-\frac{1+\epsilon}{1-\epsilon}},$ absorbing the third and fourth term on the right hand side of last inequality by the left hand side in (\ref{3.3}) follows
\begin{equation*}
\begin{aligned}
\|W^-_\lambda\|_{\tilde{L}^\infty_t(\dot{B}_{p,r}^{\frac{3}{p}-1})}&+\frac{7c}{8}\nu_+\|W^-_\lambda\|_{\tilde{L}^1_t(\dot{B}_{p,r}^{\frac{3}{p}+1})}
+C\nu_+^{-\frac{1+\epsilon}{1-\epsilon}}\|W^-_\lambda\|_{\tilde{L}^1_{t,f}(\dot{B}_{p,r}^{\frac{3}{p}-1})}\\
\le& \|W_0^-\|_{\dot{B}_{p,r}^{\frac{3}{p}-1}}+C\nu_-\|W^+_\lambda\|_{\tilde{L}^1_t(\dot{B}_{p,r}^{\frac{3}{p}+1})}.
\end{aligned}
\end{equation*}
Obviously, using (\ref{3.2}), we have
\begin{equation*}
\begin{aligned}
\|W^-_\lambda\|_{\tilde{L}^\infty_t(\dot{B}_{p,r}^{\frac{3}{p}-1})}+c\nu_+\|W^-_\lambda\|_{\tilde{L}^1_t(\dot{B}_{p,r}^{\frac{3}{p}+1})}
\le&2\|W_0^-\|_{\dot{B}_{p,r}^{\frac{3}{p}-1}}+C\nu_-\|W^+_\lambda\|_{\tilde{L}^1_t(\dot{B}_{p,r}^{\frac{3}{p}+1})}\\
\le& C\left( \|W_0^-\|_{\dot{B}_{p,r}^{\frac{3}{p}-1}}+\frac{\nu_-}{\nu_+}(\|W_0^+\|_{\dot{B}_{p,r}^{\frac{3}{p}-1}}+\nu_-)\right).
\end{aligned}
\end{equation*}
This yields, after using (\ref{3.2}) again, for all $t\in(0,\bar{T}),$
\begin{equation*}
\begin{aligned}
&\ \ \ \|W^-_\lambda\|_{\tilde{L}^\infty_t(\dot{B}_{p,r}^{\frac{3}{p}-1})}+c\nu_+\|W^-_\lambda\|_{\tilde{L}^1_t(\dot{B}_{p,r}^{\frac{3}{p}+1})}\\
\le& C\left( \|W_0^-\|_{\dot{B}_{p,r}^{\frac{3}{p}-1}}+\frac{\nu_-}{\nu_+}(\|W_0^+\|_{\dot{B}_{p,r}^{\frac{3}{p}-1}}+\nu_-)\right)
\exp\left\{C\nu_+^{-\frac{1+\epsilon}{1-\epsilon}}\int_0^t \|W^+(\tau)\|_{\dot{B}_{p,\infty}^{\frac{3}{p}-\epsilon}}^\frac{2}{1-\epsilon}d\tau
\right\}\\
\le& C\left( \|W_0^-\|_{\dot{B}_{p,r}^{\frac{3}{p}-1}}+\frac{\nu_-}{\nu_+}(\|W_0^+\|_{\dot{B}_{p,r}^{\frac{3}{p}-1}}+\nu_-)\right)
\exp\left\{C\nu_+^{-\frac{1+\epsilon}{1-\epsilon}}\|W^+\|_{\tilde{L}^\frac{2}{1-\epsilon}_t(\dot{B}_{p,r}^{\frac{3}{p}-\epsilon})}^\frac{2}{1-\epsilon}\right\}\\
\le& C\left( \|W_0^-\|_{\dot{B}_{p,r}^{\frac{3}{p}-1}}+\frac{\nu_-}{\nu_+}(\|W_0^+\|_{\dot{B}_{p,r}^{\frac{3}{p}-1}}+\nu_-)\right)
\exp\left\{C\nu_+^{-\frac{2}{1-\epsilon}}(\nu_-+\|W^+_0\|_{\dot{B}_{p,r}^{\frac{3}{p}-1}})^\frac{2}{1-\epsilon}\right\},
\end{aligned}
\end{equation*}
which implies that if we take $\eta$ small enough in (\ref{1.14}), there holds for all $t\le \bar{T}$,
$$\|W^-_\lambda\|_{\tilde{L}^\infty_t(\dot{B}_{p,r}^{\frac{3}{p}-1})}+\nu_+\|W^-_\lambda\|_{\tilde{L}^1_t(\dot{B}_{p,r}^{\frac{3}{p}+1})}\le C\eta \nu_+<\frac{\epsilon_0}{2}\nu_+.$$
Then by a standard continuous method, we get $\bar{T}=T^\star=\infty.$
\vskip .1in
The remainder is $r=1, \epsilon=0$, by a similar arguments, using  (\ref{SKP2}) for this case and let $(\epsilon,r)=(0,1)$,
$f=\|W^+\|_{\dot{B}_{p,1}^\frac{3}{p}}^2$ in (\ref{3.3}),  the desired result can be otained. Hence, we complete the proof of Theorem \ref{t1}.

\vskip .3in
\section{Proof of Theorem \ref{t2}}
\label{s4}
One can easily get the local well-posedness of (\ref{1.1}), that is, there exists a $T^\star>0$ such that
$$(u,B)\in C([0,T^\star); \chi^{-1}(\R^3))\cap L^1([0,T^\star); \chi^1(\R^3)).$$
So we suffices to  show $T^\star=\infty$.
\vskip .1in
Now, we begin the proof.   (\ref{1.21}) is indeed equal to
\begin{equation}\label{4.1}
\left(\|W_0^-\|_{\chi^{-1}}+\frac{C\nu_-}{\nu_+}(\nu_-+\|W_0^+\|_{\chi^{-1}})\right)\exp\left\{\frac{C}{\nu_+^2}(\nu_-+\|W_0^+\|_{\chi^{-1}})^2\right\}\le(2-\epsilon_0)\nu_+
\end{equation}
for some $\epsilon_0>0.$ And next we suffices to prove the desired result under (\ref{4.1}). Let $C_1,C_2\in (0,2)$ satisfying $2(2-\epsilon_0)^2<C_2(2-C_1)^2$
and
$$a=C_2 \nu_+,\ \ a_1=C_1 \nu_+,\ b\in (\frac{2(2-\epsilon_0)}{2-C_1}\nu_+,\sqrt{2C_2}\nu_+).$$
Then consider the first equation in (\ref{1.12}), by the procedure as \cite{LL}, and using interpolation inequality,  we have
\begin{equation*}
\begin{aligned}
\frac{d}{dt}\|W^+\|_{\chi^{-1}}+&\nu_+\|W^+\|_{\chi^1}\le \|W^+\cdot\nabla W^-\|_{\chi^{-1}}+\nu_-\|W^-\|_{\chi^1}\\
\le& \|W^+\|_{\chi^0}\|W^-\|_{\chi^0}+\nu_-\|W^-\|_{\chi^1}\\
\le& \|W^-\|_{\chi^0}^2\|W^+\|_{\chi^{-1}}^\frac{1}{2}\|W^+\|_{\chi^1}^\frac{1}{2}+\nu_-\|W^-\|_{\chi^1}\\
\le & \frac{1}{2a}\|W^-\|_{\chi^0}^2\|W^+\|_{\chi^{-1}}+\frac{a}{2}\|W^+\|_{\chi^1}+\nu_- \|W^-\|_{\chi^1},
\end{aligned}
\end{equation*}
which derives by integrating in time,
\begin{equation}\label{4.12}
\begin{aligned}
\|W^+\|_{L^\infty_t(\chi^{-1})}+&(\nu_+-\frac{a}{2})\|W^+\|_{L^1_t(\chi^1)}\\
\le& \frac{1}{2a}\|W^-\|_{L^2_t(\chi^0)}^2\|W^+\|_{L^\infty_t(\chi^{-1})}
+\nu_-\|W^-\|_{L^1_t(\chi^1)}+\|W_0^+\|_{\chi^{-1}}.
\end{aligned}
\end{equation}
Define
\begin{equation}\label{4.2}
\bar{T}:=\sup\left\{t\in (0,T^\star):\ \|W^-\|_{L^\infty_t(\chi^{-1})}+\nu_+\|W^-\|_{L^1_t(\chi^1)}\le b \right\}.
\end{equation}
Then we will prove $T^\star=\bar{T}=\infty$ under (\ref{4.1}). Using (\ref{4.2}), combining with (\ref{4.12}), we have
\begin{equation}\label{4.3}
(1-\frac{b^2}{2a\nu_+})\|W^+\|_{L^\infty_t(\chi^{-1})}+(\nu_+-\frac{a}{2})\|W^+\|_{L^1_t(\chi^1)}\le \|W_0^+\|_{\chi^{-1}}+\frac{b\nu_-}{\nu_+}.
\end{equation}
Following the similar way as (\ref{4.12}), one gets
$$
\frac{d}{dt}\|W^-\|_{\chi^{-1}}+\nu_+\|W^-\|_{\chi^1}
\le \frac{1}{2a_1}\|W^+\|_{\chi^0}^2\|W^-\|_{\chi^{-1}}+\frac{a_1}{2}\|W^-\|_{\chi^1}+\nu_- \|W^+\|_{\chi^1}
$$
and thanks to (\ref{4.3}),
\begin{equation*}
\begin{aligned}
\|W^-(t)\|_{\chi^{-1}}+&(\nu_+-\frac{a_1}{2})\|W^-\|_{L^1_t(\chi^1)}\\
\le& \frac{1}{2a_1}\int_0^t \|W^+\|_{\chi^0}^2\|W^-\|_{\chi^{-1}} d\tau
+\frac{\nu_-}{\nu_+-\frac{a}{2}}(\frac{b\nu_-}{\nu_+}+\|W_0^+\|_{\chi^{-1}})+\|W^{-}_0\|_{\chi^{-1}},
\end{aligned}
\end{equation*}
 with the application of Gronwall's lemma, by interpolation's inequality and (\ref{4.3}) leads
\begin{equation*}
\begin{aligned}
&\ \ \ \ \ \ \|W^-\|_{L^\infty_t(\chi^{-1})}+(\nu_+-\frac{a_1}{2})\|W^-\|_{L^1_t(\chi^1)}\\
\le& (\|W_0^-\|_{\chi^{-1}}+\frac{\nu_-}{\nu_+-\frac{a}{2}}(\frac{b\nu_-}{\nu_+}+\|W_0^+\|_{\chi^{-1}})\exp\left\{\frac{1}{2a_1}\int_0^t \|W^+\|_{\chi^0}^2 d\tau \right\}\\
\le& (\|W_0^-\|_{\chi^{-1}}+\frac{\nu_-}{\nu_+-\frac{a}{2}}(\frac{b\nu_-}{\nu_+}+\|W_0^+\|_{\chi^{-1}})\exp\left\{\frac{1}{2a_1}\|W^+\|_{L^\infty_t(\chi^{-1})}
\|W^+\|_{L^1_t(\chi^1)}\right\}\\
\le& (\|W_0^-\|_{\chi^{-1}}+\frac{\nu_-}{\nu_+-\frac{a}{2}}(\frac{b\nu_-}{\nu_+}+\|W_0^+\|_{\chi^{-1}})\exp\left\{\frac{2a\nu_+}{a_1(2a\nu_+-b^2)(2\nu_+-a_1)}
(\frac{b\nu_-}{\nu_+}+\|W^+_0\|_{\chi^{-1}})^2
 \right\}.
\end{aligned}
\end{equation*}
which indicates that there exists  constant $C$ such that
\begin{equation*}
\begin{aligned}
\|W^-\|_{L^\infty_t(\chi^{-1})}+&(1-\frac{C_1}{2})\nu_+\|W^-\|_{L^1_t(\chi^1)}\\
\le& \left(\|W_0^-\|_{\chi^{-1}}+\frac{C\nu_-}{\nu_+}(\nu_-+\|W_0^+\|_{\chi^{-1}})\right)\exp\left\{\frac{C}{\nu_+^2}(\nu_-+\|W_0^+\|_{\chi^{-1}})^2\right\}\\
\le& (2-\epsilon_0)\nu_+.
\end{aligned}
\end{equation*}
This implies that
$$\|W^-\|_{L^\infty_t(\chi^{-1})}+\nu_+\|W^-\|_{L^1_t(\chi^1)}<\frac{2(2-\epsilon_0)}{2-C_1}<b.$$
Therefore, by standard continuous method, we get $T^\star=\bar{T}=\infty.$ This concludes the proof of Theorem \ref{t2}.




\end{document}